\documentclass[11pt,twoside]{article}
\usepackage{latexsym}
\usepackage{amssymb,amsbsy,amsmath,amsfonts,amssymb,amscd,mathtools}
\usepackage{bbm,bm}
\usepackage{graphicx,color}
\usepackage{enumitem}
\usepackage{tikz,pgfplots}
\usepackage{float,url}
\usepackage{cancel}
\usepackage[colorlinks,linkcolor=red,citecolor=blue]{hyperref}
\usepackage{titlesec}
\titlespacing*{\paragraph}{0pt}{8pt}{\baselineskip}

\newcommand{\commentout}[1]{}

\newcommand {\al} {\alpha}
\newcommand {\eps}  {\varepsilon}

\newcommand {\vp} {\varphi}

\newcommand{\pd}{\partial}
\newcommand{\dt}{\pd_{t}}
\newcommand {\x} { {\bm x} }
\newcommand{\V}{{\bm V}}
\newcommand{\dx}{\nabla}

\newcommand{\dxi}{\,{\rm d}\x}
\newcommand{\intx}{\int_{\Omega}}
\newcommand{\intt}{\int_0^T}
\newcommand{\dti}{{\rm d}t}
\newcommand{\red}{\textcolor{red}}

\newcommand{\REV}[1]{\textcolor{black}{#1}}
\newcommand{\beq}{\begin{equation}}
\newcommand{\eeq}{\end{equation}}
\newtheorem{theorem}{Theorem}

\newtheorem{definition}[theorem]{Definition}

\newtheorem{proposition}[theorem]{Proposition}

\newcommand{\qed}{{ \hfill
                       {\unskip\kern 6pt\penalty 500 \raise -2pt\hbox{\vrule\vbox to 6pt{\hrule width 6pt
                       \vfill\hrule}\vrule} \par}   }}
\title{\REV{Existence, r}egularity and stability in a strongly degenerate nonlinear diffusion and haptotaxis model of cancer invasion}

\author{
Beno\^ \i t Perthame\thanks{Sorbonne Universit{\'e}, CNRS, Universit\'{e} de Paris, Inria, Laboratoire Jacques-Louis Lions UMR7598, F-75005 Paris., France 
Email : Benoit.Perthame@sorbonne-universite.fr}
\and Chiara Villa\thanks{Sorbonne Universit{\'e}, CNRS, Universit\'{e} de Paris, Inria, Laboratoire Jacques-Louis Lions UMR7598, F-75005 Paris, France. 
Email : chiara.villa.1@sorbonne-universite.fr}{$^,$\thanks{Université Paris Cité, CNRS, MAP5, Paris, F-75006, France}}
}
\date{\today}

\begin{document}
\maketitle
\pagestyle{plain}
\pagenumbering{arabic}

\begin{abstract} 
We consider a mathematical model of cancer cell invasion of the extracellular matrix (ECM), comprising a strongly degenerate parabolic partial differential equation for the cell volume fraction, coupled with an ordinary differential equation for the ECM volume fraction \REV{($0\leq\vp\leq1$)}. The model captures the intricate link between the dynamics of invading cancer cells and the surrounding ECM. First, migrating cells undergo haptotaxis, i.e., movement up ECM density gradients. Secondly, cancer cells degrade the ECM fibers by means of membrane-bound proteases. Finally, their migration speed is modulated by the ECM pore sizes, resulting in the saturation or even interruption of cell motility both at high and low ECM densities. The inclusion of the physical limits of cell migration results in two regimes of degeneracy \REV{(at $\vp=0$ and $\vp=1$)} \REV{impacting simultaneously} the nonlinear diffusion and haptotaxis terms. 

We devise specific estimates that are compatible with the strong degeneracy of the equation, focusing on the vacuum present at low ECM densities. Based on these regularity results, and associated local compactness, we prove stability of weak solutions with respect to perturbations on bounded initial data\REV{, and global existence of weak solutions in the limit of an approximate non-degenerate problem.}

\end{abstract} 

\noindent{\makebox[1in]\hrulefill}\newline 
\section{Introduction}
\label{sec:intro}

We begin with a short summary of the biological phenomena motivating the form of the equation and of the analytical efforts made for more approachable versions of the model. 

\paragraph{Biological context.} 
Migrating cells typically move following environmental clues. Notably, they may have the ability to sense environmental gradients resulting in directed movement, for instance, down density-dependent pressure gradients or up gradients of chemical concentrations (e.g., oxygen, growth factors, etc), phenomenon known as chemotaxis. 
Empirical evidence also indicates the tendency of migrating cells to move by haptotaxis~\cite{aznavoorian1990signal}, that is up density gradients of the extracellular matrix (ECM) -- i.e., the network consisting of extracellular macromolecules and elastic fibers providing cells with structural support. Haptotaxis-driven motion is particularly interesting in the context of cancer invasion, as invading cells may themselves create ECM gradients~\cite{perumpanani1999extracellular} by degrading the ECM fibers by means of diffusive or membrane-bound proteases~\cite{friedl2003tumour}.
\REV{
In particular, membrane-bound matrix metalloproteinases (MMPs) such as MT1-MMP, which are key to the activation of diffusive MMPs~\cite{almeida2025mathematical}, have been recognised to play a prominent role in the metastatic process, with experimental findings pointing towards a direct essential role of MT-MMPs in the invasion process that is independent of the action of soluble MMPs~\cite{hotary2006cancer,poincloux2009matrix}. 
} 

Independently of the environmental clues directing cell migration, the density of the ECM plays an important role in modulating the speed at which cells can migrate through 3D tissues~\cite{wolf2013physical}. On one hand, invading cells rely on the presence of ECM fibers \red{to} form focal adhesions and exert traction forces boosting them forward, and thus a lower ECM density may hinder this process. On the other hand, a denser ECM is characterised by smaller pore sizes that may slow down cell migration, by friction, or impede it altogether if the pore size is smaller than the diameter of the cell nucleus. Indeed, the proteolytic degradation of ECM fibers by cancer cells may help them achieve the optimal environmental set up for fast invasion~\cite{friedl2003tumour}.

\paragraph{State of the art.} 

A wide range of mathematical models have been proposed in the study of cancer invasion, see~\cite{sfakianakis2020mathematical} and references therein, many of which comprise a system of partial differential equations (PDEs) for the cancer cell and ECM densities, building on the framework first proposed in~\cite{anderson2000mathematical,perumpanani1999extracellular}. PDE models including haptotaxis have sparked the interest of the mathematical community, as the resulting system is not of the standard parabolic type~\cite{tao2007global} due to the nonlinear negative feedback induced by ECM degradation, and their analytical study can be more demanding than that of the classical Keller-Segel model for chemotaxis due to the non-diffusive nature of the ECM~\cite{bellomo2015toward}.

In the case of diffusive proteases, many results on global existence, boundedness and large time behaviour of the solution to the model comprising linear (or nonlinear) diffusion and haptotaxis have been obtained, following the initial results in~\cite{tao2007global,walker2007global}. These include model extensions featuring chemotaxis and ECM remodelling, building on the framework proposed by Chaplain and Lolas~\cite{chaplain2006mathematical}, as well as parabolic-elliptic variants under quasi-stationary assumptions on protease dynamics, and two-species models. We refer the interested reader to~\cite{bellomo2015toward,dai2022global,wang2020review} for an exhaustive summary of rigorous analytical results obtained up to 2022, noting that the relevant literature has since expanded~\cite{dai2023boundedness,jin2024critical,jin2024roles}. 

On the contrary, very few works have investigated the well-posedness of models solely considering membrane-bound proteases as in~\cite{perumpanani1999extracellular}, which lacks the additional regularity introduced by the protease diffusivity~\cite{marciniak2010boundedness}, mainly by Zhigun and coworkers~\cite{zhigun2016global2,zhigun2018strongly,zhigun2016global}. 
In these papers, the authors prove global existence of weak solutions in haptotaxis-driven cancer invasion models, with carefully defined diffusion and haptotactic coefficients to avoid blow up of solutions in finite time typical of Chaplain-Lolas inspired models~\cite{shangerganesh2019finite,tao2010density}. Interestingly, they focus on parabolic PDEs presenting a degeneracy of the diffusion in the absence of ECM density, although this degeneracy does not impact haptotaxis. 

The modulation of the speed of cell migration by the ECM pore size -- or, under isotropic assumptions, the ECM density or volume fraction -- has been well characterised by Preziosi and coworkers
~\cite{arduino2015multiphase,giverso2018nucleus,scianna2012cellular}. The resulting function modelling the ECM density-dependent saturation of cell migration speed may result in the degeneracy of the overall velocity~\cite{loy2020modelling}, 
both at low and high ECM density. The well-posedness of PDE models implementing such a degeneracy resulting from the physical limits of cell migration is yet to be addressed and motivates the present work. 
\REV{Naturally, the scenario of cancer invasion prominently mediated by membrane-bound MMPs, rather than diffusive ones, represents a more critical case with greater analytical challenges in dealing with the strong degeneracy of the PDE.}

\paragraph{Paper content.}
We present, in Section~\ref{sec:model}, a strongly degenerate nonlinear diffusion and haptotaxis model of cancer invasion, focusing on membrane-bound proteases, that accounts for the ECM density-dependent saturation of the overall cell migration speed. 
\REV{The degeneracy arises both at $\vp=0$ and $\vp=1$, being $\vp$ the volume fraction occupied by the ECM, and affects both the diffusion and haptotactic terms.}
In Section~\ref{sec:estimates}, we devise specific estimates compatible with the degeneracy of the equation. \REV{In this section, we focus on the degeneracy at 0 under a stricter bound on the initial ECM volume fraction, that allows us to avoid the degeneracy at 1, which is introduced for technical reasons to obtain stronger gradient controls. } 
Based on the regularity provided by these estimates, we obtain some local compactness results in Section~\ref{sec:compactness} and prove stability of weak solutions with respect to perturbations of bounded initial data in Section~\ref{sec:stability}. 
\REV{Existence of weak solutions is addressed in Section~\ref{sec:existence}, where solutions to the degenerate problem are obtained in the limit of an approximate non-degenerate problem -- see Figure~\ref{f1} for a schematic illustration.} 
We conclude with a brief discussion in Section~\ref{sec:discuss}\REV{, in which the biological rationale behind the stricter bound on the initial ECM volume fraction is also addressed.}

\section{The model}
\label{sec:model}

Consider a population of cancer cells invading the surrounding tissue. For the purpose of our work, we here only take into account the role played by the ECM in this process, i.e. we ignore other cell populations and abiotic factors for simplicity. Cells are assumed to proliferate in the limits imposed by the availability of free space, and die due to compression in the lack of it.  Spatial movement of cells is assumed to be driven by haptotaxis and cell density-induced pressure gradients, i.e. cells move up ECM density gradients and down cell density gradients, within the physical limits imposed by ECM pore sizes.  Finally, cancer cells are assumed to have the ability to degrade the ECM and increase the availability of free space to move and proliferate in.
%

Let $\psi(t,\x)\geq0$ and $\vp(t,\x)\geq0$ respectively denote the volume fraction of a population of cancer cells and that of the ECM components at time $t\in\mathbb{R}_{\geq0}$ and spatial position $\x\in\Omega\subseteq\mathbb{R}^d$ ($d=1,2,3$). Then $\psi(t,\x)$ and $\vp(t,\x)$ satisfy the following system of PDEs for $t>0$ and $\x\in\Omega$
\begin{equation}
\label{eq}
\begin{cases}
\dt \psi +  \dx \cdot \left[\, \psi\,  {\V}(\psi,\vp)\,\right] = \psi  R(\psi,\vp) \,,\\[5pt]
\dt \vp = -\gamma \vp \psi \,.
\end{cases}
\end{equation}
The second term on the left-hand-side of Eq.~\eqref{eq}$_1$ models cell motion with velocity ${\V}$ given by
\begin{equation}\label{def:V}
    {\V}(\psi,\vp):= S(\vp)\,\big(\alpha \dx \vp\,-D \dx \psi \big)\,,
\end{equation}
with haptotactic coefficient $\alpha>0$, nonlinear diffusion coefficient $D>0$ and where $S(\vp)$ models the modulation of the motility of cancer cells according to healthy tissue permeability, and is therefore a function of the local ECM volume fraction.  
Building on the work presented in~\cite{giverso2018nucleus,loy2020modelling}, we take
\begin{equation}\label{def:S}
    S(\vp):= \vp(1-\vp)\,,
\end{equation}
capturing the saturation of cell motility when the ECM volume fraction is too low (i.e.  cells lack proper structural support of ECM fibers to move quickly) or too high (i.e. cells are slowed down by the physical constraints imposed by small ECM pores).  The term on the right-hand-side of Eq.~\eqref{eq}$_1$ models cell proliferation and death at a net rate $R(\psi,\vp)$ which, following classical volume exclusion assumptions~\cite{anderson2000mathematical,chaplain2006mathematical,perumpanani1999extracellular}, we take to be an increasing function of the available space, i.e.
\begin{equation}\label{def:R}
R(\psi,\vp):=  R_0 \,(1-\vp-\psi)\,,
\end{equation}
where $R_0>0$ is the maximum net proliferation rate. Finally, Eq.~\eqref{eq}$_2$ models ECM degradation by cancer cells at a rate $\gamma>0$.

\begin{figure}
    \centering    \includegraphics[width=\linewidth]{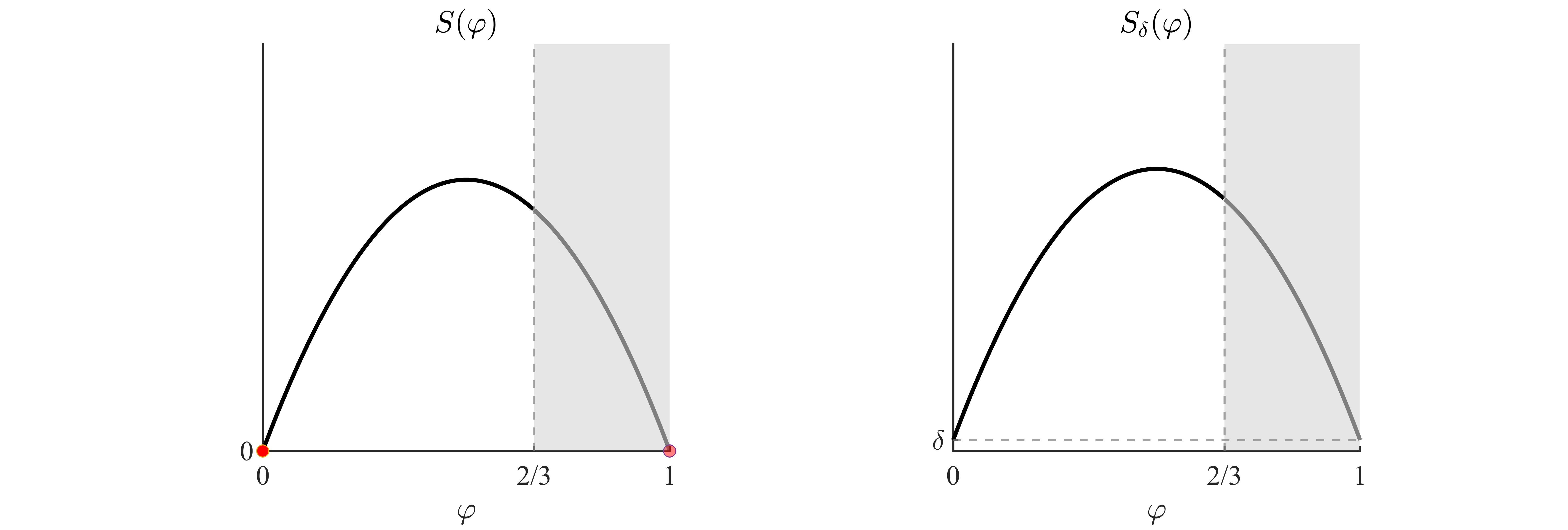}
    \caption{\REV{Schematic illustration of the modulation of the velocity according to the ECM volume fraction $\vp$. 
    Left: behaviour of the function $S(\vp)$ defined in~\eqref{def:S} marking the strong degeneracy of the velocity $\bm V$ at $\vp=0$ and $\vp=1$ (in red) along with the region excluded under assumption~\eqref{as:2} (in gray). Right: corresponding illustration of the perturbed function $S_\delta(\vp)=\delta+S(\vp)>0$ used in the perturbed non-degenerate problem considered to obtain existence in the limit $\delta\to0$ in Section~\ref{sec:existence}.}}
    \label{f1}
\end{figure}

System~\eqref{eq} is complemented with the initial conditions
\begin{equation}\label{ic}
\psi(0,\x)=\psi^0(\x)\qquad \text{and} \qquad \vp(0,\x)=\vp^0(\x)\,,
\end{equation}
which are assumed to satisfy the following properties:
\begin{align} \label{as:1}
0\leq \vp^0(\x) \leq 1, \quad \intx \vp^0 (\x) \dxi  < \infty, \quad 0 \leq \psi^0(\x) , \quad  \intx \left(1+\frac{|\x|^2}{2} + \psi^0(\x)  \right) \psi^0 (\x) \dxi < \infty. 
\end{align}
After initial regularity estimates in Section~\ref{sec:estimates:1}, we proceed with the additional stronger assumptions
\begin{equation} \label{as:2}
0\leq \vp^0(\x) \leq \frac 2 3, \qquad |\nabla \sqrt \vp^0 |\in L^2(\Omega). 
\end{equation}
\REV{The major difficulty to analyse this system is the degeneracy when $\varphi(t,\x)=0$. The degeneracy for  $\vp(t,\x)=1$ can be avoided -- see also gray area in Figure~\ref{f1} -- because $\vp$ is decreasing in time and the assumption~\eqref{as:2}, which is fundamental in our analysis, is propagated with time.}

In the following we take $\Omega\equiv\mathbb{R}^d$, but analogous results can be obtained for $\Omega\subset\mathbb{R}^d$ complementing~\eqref{eq}$_1$ with no-flux boundary conditions on $\partial\Omega$.


\section{Regularity estimates \REV{and inclusions}}
\label{sec:estimates} 

\subsection{Regularity results for {\boldmath $\vp^0\leq1$}}\label{sec:estimates:1}

\begin{proposition}[Elementary a priori \REV{bounds}]\label{prop1}
Let $\Omega\equiv\mathbb{R}^d$, and let \REV{$(\psi,\varphi)$, with} $\psi\geq0$ and $\vp\geq0$\REV{, be a weak solution built in Section~\ref{sec:existence}} of~\eqref{eq} and definitions~\eqref{def:V}-\eqref{def:R}, complemented with initial conditions $\psi(0,\x)=\psi^0(\x)$ and $\vp(0,\x)=\vp^0(\x)$ satisfying assumptions~\eqref{as:1}.  Then for all $t\in[0,\infty)$ and $\x\in\Omega$ we have that 
\begin{equation} \label{est:phiLinf}
\vp(t,\x) \leq \vp^0(\x).
\end{equation}
In addition, \REV{that the following quantities are bounded in the following spaces} for all $T>0$: 
\begin{align}
    \psi\in &\; L^\infty \left((0,T); L^1(\Omega)\cap L^2(\Omega)\right) , \label{est:psiL1} \\[5pt]
     \frac{|\x|^2}{2} \psi \in &\; L^\infty((0,T); L^1(\Omega)), \label{est:psiX2} \\[5pt]
      \psi |\ln(\psi)|\in &\; L^\infty((0,T); L^1(\Omega)), \label{est:psiLog} \\[5pt]
    \psi\in &\; L^3\left((0,T)\times \Omega\right),\label{est:psiL3}\\[5pt]
     \psi \vp(1-\vp)|\al \dx \vp -D \dx \psi|^2\in &\;L^1\left((0,T)\times \Omega\right).\label{est:gradient}
\end{align}

\end{proposition}

\begin{proof}\\
\REV{The proposition follows from Section~\ref{sec:existence} but, in preparation, we give here the main calculations which hold for strong solutions.}
\\
\underline{Step 1. Upper bound on $\vp$.} The bound~\eqref{est:phiLinf} on $\vp$ comes directly from integrating Eq.~\eqref{eq}$_2$ and using the non-negativity of $\psi$, yielding that $\vp$ can only decrease in time. Therefore $\vp\in L^\infty((0,\infty)\times\Omega)$.
\\

\noindent \underline{Step 2. $L^1$, $L^2$ and $L^3$ bounds on $\psi$, the energy and initial gradients control.} Integrating Eq.~\eqref{eq}$_1$, in space and time, and applying the Gronwall lemma we obtain
\begin{equation}\label{est:psiL1a}
\intx \psi(t,\x) \dxi \leq \intx\psi^0(\x) \dxi \, e^{R_0 t}
\end{equation}
which yields $\psi\in L^\infty((0,T); L^1(\Omega))$ thanks to~\eqref{as:1}. The same integration, knowing~\REV{\eqref{est:psiL1a}}, yields 
\begin{equation} \label{est:psiL2}
 R_0 \intt\hspace{-5pt} \intx\psi(t,\x)^2 \dxi \leq C_1(T)
\end{equation}
for some constant $C_1(T)$, i.e. $\psi\in L^2((0,T)\times \Omega)$ as $R_0>0$.  Since the problem has gradient flow structure we use the energy
\begin{align*}
\frac{d}{dt} \intx\psi\left(\frac D 2 \psi - \al \vp\right)\dxi &=  \intx \left[ \partial_t \psi ( D  \psi - \al \vp) -\al \psi \partial_t \vp\right]\dxi
\\
&=-  \intx \psi \vp(1-\vp)|\al \dx \vp -D \dx \psi|^2\dxi +\intx \psi R  ( D  \psi - \al \vp)\dxi+\intx\al \gamma \psi^2 \vp\dxi .
\end{align*}
From this, using definition~\eqref{def:R} for $R$ and the bound~\eqref{est:phiLinf} with assumption~\eqref{as:1} on $\vp^0$, some simple algebra leads to
\begin{equation*}
    \frac{d}{dt} \intx\psi\left(\frac D 2 \psi - \al \vp\right)\dxi \leq C_2 \intx\psi\left(\frac D 2 \psi - \al \vp\right)\dxi +\alpha (C_2+R_0) \intx\psi\dxi ,
\end{equation*}
where $C_2:= \frac{2}{D}\left( DR_0+\alpha R_0+\alpha\gamma\right)$. Then, using\REV{~\eqref{est:psiL1a}}, assumption~\eqref{as:1} and the Gronwall lemma, we conclude that
\begin{equation}\label{est:energy}
     \intx\psi(t,\x)\left(\frac D 2 \psi(t,\x) - \al \vp(t,\x)\right)\dxi \leq C_3(t) 
\end{equation}
for some constant $C_3(t)$. From the same equality for the energy, integrating in time for $\REV{t}\in[0,T]$ and using~\eqref{est:energy}, we obtain 
\begin{equation*}
    D R_0 \displaystyle \intt\hspace{-5pt} \intx\psi (t,\x)^3 \dxi \dti \leq C_4(T) \quad \text{and} \quad \intx\psi \vp(1-\vp)|\al \dx \vp -D \dx \psi|^2\dxi \dti \leq C_5(T),
\end{equation*}
i.e. respectively estimates~\eqref{est:psiL3}, since $D,R_0>0$, and~\eqref{est:gradient}. From~\eqref{est:energy} and\REV{~\eqref{est:psiL1a}} we also have
\begin{equation*}
    \sup_{t\leq T}\intx\psi (t, \x)^2 \dxi \leq C_6(T),
\end{equation*}
i.e. $\psi\in L^\infty((0,T); {L^{\REV{2}}}(\Omega))$. This concludes estimate~\eqref{est:psiL1}. 
\\

\noindent \underline{Step 3. The second moment.} Then, the second moment can be estimated as
\begin{align*}
\frac{d}{dt} \intx\frac{|\x|^2}{2} \psi(t,\x) \dxi &= \intx |\x|\, \psi \vp (1-\vp) |\al \dx \vp -D \dx \psi| \dxi+ \intx \frac{|\x|^2}{2}  \psi R_0(1-\psi-\vp) \dxi
\\
&\leq \sqrt{ \intx |\x|^2  \psi \vp (1-\vp) \dxi \;  \intx\psi \vp (1-\vp) |\al \dx \vp -D \dx \psi|^2\dxi }\; +
\intx\frac{|\x|^2}{2}  \psi R_0 \dxi \\
&\leq   \intx\frac{|\x|^2}{2} \psi \vp (1-\vp) \dxi+ \frac 12  \intx\psi \vp (1-\vp) |\al \dx \vp -D \dx \psi|^2\dxi  +
\intx\frac{|\x|^2}{2}  \psi R_0 \dxi ,
\end{align*}
where we applied the Cauchy-Schwarz and the AM-GM inequalities to the first term. From this, thanks to the Gronwall lemma, estimate~\eqref{est:gradient} and assumption~\eqref{as:1}, we conclude~\eqref{est:psiX2}.\\

\noindent \underline{Step 4. The entropy.} Following the strategy presented in~\cite{doumic2024multispecies}, we obtain~\eqref{est:psiLog} as a consequence of \eqref{est:gradient} and \eqref{est:psiX2} by writing
\begin{equation*}
\intx \psi  | \ln(\psi) | \dxi =\int_{\psi >1}\psi  | \ln(\psi) | \dxi + \int_{1 \geq \psi \geq e^{-|\x|^2}}\psi  | \ln(\psi) | \dxi + \int_{\psi  <  e^{-|\x|^2}}\psi  | \ln(\psi) | \dxi \,.
\end{equation*}
For the first term we have
\[
\int_{\psi >1}\psi  | \ln(\psi) | \dxi \leq \int_{\psi >1} \psi^2\dxi\,,
\]
which is controlled thanks to \eqref{est:gradient}. For the second term, $| \ln(\psi) |$ being decreasing, we have
\[
 \int_{1 \geq \psi \geq e^{-|\x|^2}} \psi  | \ln(\psi) | \dxi \leq  \int_{1 \geq \psi \geq e^{-|\x|^2}} \psi (t,\x) |\x|^2 \dxi \,,
\]
which is controlled thanks to  \eqref{est:psiX2}. Finally, for the third term,  the inequality $\psi  | \ln(\psi) | \leq   \sqrt \psi $  gives
\[
 \int_{\psi  <  e^{-|\x|^2}}\psi  | \ln(\psi) | \dxi  \leq  \int_{\psi  <  e^{-|\x|^2}} \sqrt \psi \dxi \leq \int_{\psi  <  e^{-|\x|^2}} e^{-|\x|^2/2} \dxi \,,
 \]
 which is also under control. Altogether we obtain~\eqref{est:psiLog}. \hfill\qed
\end{proof}


\subsection{Additional regularity results under assumptions~\eqref{as:2}}\label{sec:estimates:2}

\begin{proposition} [Further \red{bounds}] \label{prop2}
Let $\Omega\equiv\mathbb{R}^d$, and let 
\REV{$(\psi,\varphi)$, with} $\psi\geq0$ and $\vp\geq0$\REV{, be a weak solution built in Section~\ref{sec:existence}}
of~\eqref{eq} and definitions~\eqref{def:V}-\eqref{def:R}, complemented with initial conditions $\psi(0,\x)=\psi^0(\x)$ and $\vp(0,\x)=\vp^0(\x)$ satisfying assumptions~\eqref{as:1} and~\eqref{as:2}. 
We then have the following \red{bounds} for all $T>0$: 
\begin{align}
    | \nabla \sqrt{\vp} | \; \in &\; L^\infty((0,T); L^2(\Omega)),\label{est:compact:phi}
    \\[5pt]
    \vp   |\dx \psi |^2\in &\; L^1((0,T)\times\Omega) , \label{est:FullGrad:psi}
    \\[5pt]
   \psi |\nabla \sqrt \vp|^2\in &\; L^1((0,T)\times\Omega) , \label{est:FullGrad3:phi}
   \\[5pt]
    \psi \vp(1-\vp)  |\dx \psi |^2\in &\; L^1((0,T)\times\Omega) . \label{est:FullGrad2:psi}
\end{align}
Finally, we also have a bound for all $T>0$
\begin{equation}\label{est:compact:psi}
\big| \nabla \big( \vp^{\frac 1 2} \psi\big) \big|  \in L^{2}+ L^{\frac 32}( (0,T)\times \Omega).
 \end{equation} 
\end{proposition}

We cannot obtain compactness estimates on $\psi$ because of the degeneracy when $\vp$ vanishes locally. This leads us to a gradient estimate on the quantity in \eqref{est:compact:psi} which vanishes with $\vp$ and is enough for our later purpose. \\

\begin{proof}\\
\underline{Step 1. A stronger gradient control.} We first write 
 \begin{align*}
 \partial_t \left[\frac{|\dx \vp|^2}{2} (1-\vp)\right]& = (1-\vp)  \dx \vp  \cdot \left(  \partial_t \dx \vp\right)-  \frac{|\dx \vp|^2}{2}\partial_t \vp
 \\
 &=-  \gamma(1-\vp)   \left[ \psi |\dx \vp|^2 +\vp  \dx \vp \cdot\dx \psi \right] + \gamma \frac{|\dx \vp|^2}{2} \vp \psi\\
 &=-\gamma\left(1- \frac 3 2 \vp \right)\psi |\dx\vp|^2  - \gamma (1-\vp) \vp  \dx \vp\cdot\dx \psi \quad \leq  - \gamma (1-\vp) \vp  \dx \vp\cdot\dx \psi  ,
 \end{align*} 
retrieving the inequality thanks to assumption \eqref{as:2} and the bound~\eqref{est:phiLinf}. Secondly, we use the entropy
\begin{equation}\label{est:entropy}
\frac{d}{dt} \intx \psi  \ln(\psi) \dxi = \intx\dx \psi \cdot\left[ (1-\vp) \vp (\al \dx \vp -D \dx \psi)\right]\red{\dxi} +  \intx\psi (1+ \ln(\psi)) R_0(1-\psi-\vp)\red{\dxi},
\end{equation} 
and thus, combining these two equalities, obtain 
\begin{align*}
\frac{d}{dt} \intx &\bigg[\frac \al \gamma \frac{|\dx \vp|^2}{2} (1-\vp) +  \psi  \ln(\psi) \bigg]\dxi 
\\
&\leq  \intx (1-\vp) \vp \left[ - \al \dx \vp\cdot\dx \psi  +  \al \dx \vp\cdot  \dx \psi  - D | \dx \psi |^2  \right] \dxi +  \intx\psi (1+ \ln(\psi)) R_0(1-\psi-\vp)\dxi
\\
&\leq - D \intx (1-\vp) \vp |\dx \psi |^2\dxi+ R_0 \intx \left[\psi + (1-\vp)\psi\ln(\psi)) -\psi^2 -\psi^2\ln(\psi) - \vp\psi \right] \dxi
\\
&\leq - D \intx (1-\vp) \vp |\dx \psi |^2\dxi+ R_0  \intx \psi \dxi +R_0 \intx  \psi\ln(\psi) \dxi - R_0 \intx\psi^2\ln(\psi)  \dxi
\\
&\leq - D \intx (1-\vp) \vp |\dx \psi |^2\dxi+ R_0  \intx \psi \dxi +R_0 \int_{\psi\geq1}  \psi|\ln(\psi)| \dxi + R_0 \int_{\psi<1}\psi^2|\ln(\psi)|  \dxi
\\
&\leq - D \intx (1-\vp) \vp |\dx \psi |^2\dxi+ R_0  \intx \psi \dxi +R_0 \intx  \psi|\ln(\psi)| \dxi  \,.
\end{align*} 
Since the second and third terms are controlled by estimates~\eqref{est:psiL1} and~\eqref{est:psiLog} respectively, integrating in time we conclude 
\begin{equation}\label{est:changes}
{\frac \al \gamma \intx  \frac{|\dx \vp|^2}{2} (1-\vp)\dxi \leq C_7(T)}, \qquad\\
\end{equation} 
and
$$
 D \intt\hspace{-5pt} \intx (1-\vp) \vp  |\dx \psi |^2\dxi \dti \leq C_8(T),$$
the latter corresponding to estimate~\eqref{est:FullGrad:psi} since $\alpha,\,\gamma,\,D>0$.\\

\noindent \underline{Step 2. Gradient estimate for $\sqrt{\vp}$.} 
Building on the ideas presented in~\cite{zhigun2018strongly,zhigun2016global}, at this stage we notice that $2\dt \sqrt \vp=- \sqrt \vp \psi $ and thus
\begin{align*}
\frac{d}{dt}\intx|\nabla \sqrt \vp |^2\dxi & = 2  \intx \nabla \sqrt \vp\cdot \nabla \dt \sqrt \vp \dxi = - \gamma  \intx\nabla \sqrt \vp\cdot \nabla ({\sqrt \vp \psi})\dxi\\
&=- \gamma  \intx\left[ |\nabla \sqrt \vp|^2 \psi +\sqrt \vp\,  \nabla \sqrt \vp \cdot \nabla \psi\right]\dxi
\\
&\leq - \gamma  \intx|\nabla \sqrt \vp|^2 \psi \dxi+\frac {\gamma}{2} \intx|\nabla \sqrt \vp |^2\dxi + \frac {\gamma}{2} \intx  \vp  | \nabla  \psi |^2\dxi \\
&\leq - \gamma  \intx|\nabla \sqrt \vp|^2 \psi \dxi+\frac {\gamma}{2} \intx|\nabla \sqrt \vp |^2\dxi + \frac {3\gamma}{2} \intx (1-\vp) \vp  | \nabla  \psi |^2\dxi
\end{align*} 
where we have used the upper bound~\eqref{est:phiLinf} under assumption~\eqref{as:2} in the last step.  
Using the estimate \eqref{est:FullGrad:psi} and the Gronwall lemma, we conclude that for $t\leq T$
\begin{equation*} 
\intx|\nabla \sqrt \vp |^2 \dxi \leq C_9(T) , \qquad \quad \intt\hspace{-5pt} \intx\psi |\nabla \sqrt \vp|^2 \dxi \dti \leq C_{10}(T)\,,
\end{equation*} 
i.e. estimates~\eqref{est:compact:phi} and~\eqref{est:FullGrad3:phi}.
Moreover, using the second estimate we just obtained on~\eqref{est:gradient} we have 
\begin{equation*}
\intt\hspace{-5pt} \intx\psi  \vp (1-\vp) |\nabla \psi|^2  \dxi \dti \leq C_{11}(T) ,
\end{equation*} 
concluding~\eqref{est:FullGrad2:psi}.\\

\noindent \underline{Step 3.  Gradient estimate for $\vp^{\frac 1 2}\psi$.}
Finally, we expand 
$$
\nabla \big( \vp^{\frac 1 2} \psi\big) =\vp^{\frac 1 2}\nabla \psi 
+\psi^{\frac 1 2}\; \psi^{\frac 1 2} \nabla \sqrt \vp \, .
$$
The first term on the right belongs to  $L^2((0,T)\times\Omega)$ thanks to \eqref{est:FullGrad:psi}. The second term belongs to $L^{\frac 3 2}((0,T)\times\Omega)$ given that
\[
\intt\hspace{-5pt} \intx \left|\psi^{\frac 1 2}\; \left( \psi^{\frac 1 2}\nabla \sqrt \vp \right)\right|^{\frac 3 2 } \dxi\dti \leq \left(\intt\hspace{-5pt} \intx \psi^3 \dxi\dti 
 \right)^{\frac 1 4}\left(\intt\ \intx \vp \left|\nabla \psi\right|^{2 } \dxi\dti \right)^{\frac 3 4}\,,
\]
obtained by applying the H\"older inequality (with $p=4$, $q=4/3$), is bounded thanks to~\eqref{est:psiL3} and~\eqref{est:FullGrad3:phi}.  \hfill \qed
 \end{proof}

\section{Local compactness}
\label{sec:compactness}

Next, we require space and time compactness. 
Under the stronger assumption~\eqref{as:2} on the upper bound of $\vp^0$, we obtained local compactness in space of $\sqrt{\vp}$ from~\eqref{est:compact:phi} (and of $\vp$ since $ |\dx \vp |^2 = 4 | \dx \sqrt \vp |^2 \vp $ is also bounded) and of $\vp^{\frac 1 2} \psi$
from~\eqref{est:compact:psi}.

We now focus on time compactness. First of all, since $\displaystyle{ | \partial_t \sqrt{\vp} | = \frac{\gamma}{2}\sqrt{\vp}\psi}$, from estimates~\eqref{est:phiLinf} and~\eqref{est:psiL3} we conclude that 
\begin{equation}\label{est:timecompact:vp}
    | \partial_t \sqrt{\vp} | \in L^3((0,T)\times\Omega)\,,
\end{equation}
i.e., local time compactness for $\sqrt{\vp}$ (and for $\vp$, similarly as for space compactness). 

Secondly, we use the Aubin-Lions-Simon lemma~\cite[Theorem 5]{simon1986compact} to gain time compactness of $\vp^{\frac 1 2} \psi$. We combine the equations on $\vp$ and $\psi$ in~\eqref{eq}, to get 
\begin{equation}\label{eqbis}
\dt (\vp^{\frac 1 2} \psi) =-  \dx \cdot \left[\, \vp^{\frac 1 2} \psi\,  {\V}\,\right]+ \;  \psi \dx \vp^{\frac 1 2}\cdot {\V} + \psi  \vp^{\frac 1 2} R -\frac \gamma 2 \vp^{\frac 1 2} \psi^2 \,,
\end{equation}
with ${\V}\equiv {\V}(\psi,\vp)$ and $R\equiv R(\psi,\vp)$ given by~\eqref{def:V}-\eqref{def:R}. 
From this, to apply the Aubin-Lions-Simon lemma, it remains to  estimate the right-hand-side of Eq.~\eqref{eqbis} as the divergence of a term bounded in $L^1_{\rm loc}$ and other terms bounded in $L^1_{\rm loc}$. To get this, 
\REV{we first prove that}
$\vp^{\frac 1 2} \psi\,  {\V}(\psi,\vp)\in L^1( (0,T)\times\Omega )$. \REV{This follows from}
\begin{align*}
\intt\hspace{-5pt} \intx \vp^{\frac 1 2} \psi\,  |{\V}(\psi,\vp)|\dxi\dti &=\intt\hspace{-5pt} \intx \left( \vp^{\frac 1 2}\vp^{\frac 1 2}(1-\vp)^{\frac 1 2}\psi^{\frac 1 2}\right) \left(\psi^{\frac 1 2}\vp^{\frac 1 2}(1-\vp)^{\frac 1 2} |\alpha \dx \vp -D\dx \psi | \right) \dxi\dti
\\
&\leq \left[\intt\hspace{-5pt} \intx \psi \vp^{2} (1-\vp)\dxi\dti
\intt\hspace{-5pt} \intx \psi\vp(1-\vp)\left|\alpha \nabla \vp- D \nabla \psi\right|^2\dxi\dti \right]^\frac{1}{2}
\end{align*}
after using the Cauchy-Schwarz inequality. The first integral is bounded by~\eqref{est:psiL1} thanks to bound~\eqref{est:phiLinf}, 
and the second one is bounded by estimate~\eqref{est:gradient}\REV{, allowing us to conclude a uniform bound in the space
\begin{equation}
    \vp^{\frac 1 2} \psi\,  |{\V}(\psi,\vp)|\in L^1_{loc}((0,T)\times\Omega).
\end{equation}}

The second term on the right-hand side of Eq.~\eqref{eqbis} is bounded in $L^1$ since  we have
\begin{align*}
    \intt\hspace{-5pt}\intx  \psi |\dx &\vp^{\frac 1 2}\cdot {\V}(\psi,\vp)| \dxi\dti 
    = \intt\hspace{-5pt} \intx  \big |\dx \sqrt{\vp} \cdot \left[ \psi\vp(1-\vp)(\alpha \dx \vp - D\dx\psi) \right]\big |\dxi\dti
    \\
    &\leq \left[\intt\hspace{-5pt} \intx \vp (1-\vp)\psi\left|\dx\sqrt{\vp}\right|^2 \dxi\dti  
    \intt\hspace{-5pt} \intx  \psi\vp(1-\vp) \left|\alpha \dx \vp -D\dx \psi\right|^2 \dxi\dti \right]^{\frac 1 2} \hspace{-3pt} ,
\end{align*}
again after using the Cauchy-Schwarz inequality. The first integral is bounded by estimate~\eqref{est:FullGrad3:phi} thanks to bound~\eqref{est:phiLinf} and under assumption~\eqref{as:2}. The second integral is bounded by estimate~\eqref{est:gradient}.
Finally, the last two terms in Eq.~\eqref{eqbis} are bounded in $L^1$ thanks to~\eqref{est:phiLinf} and~\eqref{est:psiL3}. \REV{These bounds conclude time compactness proof}. 

\section{Stability with respect to initial perturbations}
\label{sec:stability} 

 We conclude by proving stability of weak solutions under initial perturbations, through the convergence of subsequences of solutions corresponding to initially perturbed data.

\begin{definition} \label{def:weak} We say that $\psi(t,\x)$ and $\vp(t,\x)$, satisfying all properties of Propositions~\ref{prop1} and~\ref{prop2}, are weak solutions of system~\eqref{eq} under initial conditions $\psi^0(\x)$ and $\vp^0(\x)$ satisfying properties~\eqref{as:1} if, for all $T>0$ and all test functions $\chi(t,\x)\in C_c^\infty ([0,T]\times\Omega)$ with $\chi(T,\x)\equiv0$,
\[
 -\int_0^T \hspace{-5pt}\int_\Omega  \psi \dt \chi\dxi\dti =  \int_\Omega  \psi^0 \chi(0,\x)\dxi + \int_0^T \hspace{-5pt}\int_\Omega 
 \psi \nabla \chi\cdot {\V}(\psi,\vp) \dxi\dti + 
\intt\hspace{-5pt}\intx \chi \psi  \bar R(\psi,\vp)\dxi\dti
\]
\[
\vp(t,x)= \vp^0(x)-\gamma\int_0^t \vp \psi(s,x) ds, \qquad \forall t\geq 0, \; x \in \Omega,
\]
where ${\V}(\psi,\vp)$ and $\bar R(\psi,\vp)\equiv R(\psi,\vp)$ are given by~\eqref{def:V}-\eqref{def:R}. 
\end{definition}

It is necessary to introduce a possible weak limit $\bar R$ because of the degeneracy of Eq.~\eqref{eq}. For example, if $\vp^0\equiv 0$, then initial oscillations in $\psi^0_\eps$ are not regularized and the weak limit of $\psi_\eps R(\psi_\eps, \vp^0)=-R_0 \psi_\eps^2$ will \red{not} be $-R_0\psi^2$.  

\begin{theorem} [Weak limits of solutions]
Consider a family of initial data $\vp^0_\eps$, $\psi^0_\eps$ which satisfies, uniformly in $\eps $, the initial property~\eqref{as:1}, \eqref{as:2}. Let $\vp_\eps$, $\psi_\eps$ be the solutions of~\eqref{eq}, in the sense of Definition~\ref{def:weak}, corresponding to initial conditions $\vp^0_\eps$, $\psi^0_\eps$. Then we have that for all $T>0$ and all balls $B\subset\Omega$,
\[
\vp_\eps(t,\x)\to \vp(t,\x) \quad \text{strongly in} \quad L^{p}((0,T)\times B), \quad 1\leq p < \infty  \,,
\]
\[
\psi_\eps \rightharpoonup \psi \quad \text{weakly in} \quad L^{3}((0,T)\times \Omega) \, ,
\]
\[
\vp^{\frac 1 2 }_\eps(t,\x) \psi_\eps (t,\x)\to \vp^{\frac 1 2} \psi \;\;\; \text{strongly in} \;\;\; L^{p}((0,T)\times B) , \;\; 1\leq p < 3 \,,
\]
\[
\psi_\eps \to \psi \quad \text{a.e. in the set} \quad  \{\vp>0\}) \, ,
\]
and the limit functions $\vp, \; \psi$ satisfy Eq.~\eqref{eq} in the sense of Definition~\ref{def:weak}, with $\bar R = R$ on the set $\{\vp>0\}$. Moreover, if we also have $0<\vp^0:= \text{w-lim}\, \vp^0_\eps$, then 
\[
\psi_\eps \to \psi \quad \text{strongly in} \quad L^{p}((0,T)\times B), \quad 1\leq p <3 \, ,
\]
 and the limit functions $\vp, \; \psi$ satisfy Eq.~\eqref{eq} in the sense of Definition~\ref{def:weak}, with $\bar R \equiv R$ everywhere. 
\end{theorem}

\begin{proof}

\noindent \underline{Step 1.  Subsequence extraction and pointwise convergence.}
     Thanks to the local compactness results proved in Section~\ref{sec:compactness}, we may extract a subsequence such that for all $T>0$ and all balls $B\subset\Omega$,
\[
\vp_\eps(t,\x)\to \vp(t,\x) \quad \text{strongly in} \quad L^{p}((0,T)\times B), \quad 1\leq p < \infty \qquad \text{and}\quad |\nabla \vp |\in L^\infty((0,T); L^2(\Omega))\,
\]
and this subsequence also converges almost everywhere. For $\psi_\eps$, we cannot argue so simply. On the one hand, from the bound~\eqref{est:psiL3}, we may extract another subsequence such that
\[
\psi_\eps \rightharpoonup \psi \quad \text{weakly in} \quad L^{3}((0,T)\times \Omega) \, .
\]
On the other hand, the local compactness argument in Section~\ref{sec:compactness}, and the $L^3$ bound on $\psi_\eps$, give the strong convergence
\[
\vp^{\frac 1 2}_\eps(t,\x) \psi_\eps (t,\x)\to \vp^{\frac 1 2} \psi \qquad \text{a.e. and thus strongly in} \;\;\; L^{p}((0,T)\times B) , \;\; 1\leq p < 3 \, ,
\]
and the weak convergence
\[
\nabla (\vp_\eps^{\frac 1 2} \psi_\eps) \rightharpoonup  \nabla (\vp^{\frac 1 2} \psi) \in L^{3/2}((0,T)\times B)\, .
\]
Notice that, by weak-strong convergence, we know indeed that the limit is $ \vp^{\frac 1 2} \psi$. 
\REV{In the set $\{\vp(t,\x)>0\}$ we know that $\vp_\eps(t,\x)\to \vp(t,\x)$ almost everywhere pointwise and thus we also have the strong convergence   }
\[
\psi_\eps (t,\x)\to  \psi (t,\x) \quad \text{strongly in} \quad L^{p}((0,T)\times B\cap\{\vp>0\}) , \quad 1\leq p <3 \, .
\]
    Finally, if $0<\vp^0(\x)$, then from Eq.~\eqref{eq}$_2$ (which holds in the limit again by weak-strong convergence, see Step 2) we see that we have $\vp(t,\x)>0$ for all $t\in[0,T]$ and thus conclude that $\psi_\eps $ converges almost everywhere and
\[
\psi_\eps (t,\x)\to  \psi (t,\x) \quad \text{strongly in} \quad L^{p}((0,T)\times B), \quad 1\leq p <3 \,.
\]

\noindent \underline{Step 2. Passing to the limit in Eq.~\eqref{eq}.} We now show that the limit functions $\vp, \; \psi$ satisfy Eq.~\eqref{eq} in the sense of Definition~\ref{def:weak}, by passing to the limit in the equation.

Multiplying both sides of the equations in system~\eqref{eq} for $\vp_\eps, \; \psi_\eps$ by a test function $\chi(t,\x)\in C_c^\infty ([0,T]\times\Omega)$ with $\chi(T,\x)\equiv0$ and integrating by parts we obtain
\begin{equation}\label{eq:weak:psi}
     -\int_0^T \hspace{-5pt}\int_\Omega  \psi_\eps \dt \chi\dxi\dti =  \int_\Omega  \psi_\eps^0 \chi(0,\x)\dxi + \int_0^T \hspace{-5pt}\int_\Omega 
 \psi_\eps \nabla \chi\cdot {\V}(\psi_\eps,\vp_\eps) \dxi\dti + 
\intt\hspace{-5pt}\intx \chi \psi_\eps  R(\psi_\eps,\vp_\eps)\dxi\dti
\end{equation}
and 
\begin{equation}\label{eq:weak:vp}
    \vp_\eps(t,x)= \vp_\eps^0(x)-\gamma\int_0^t \vp_\eps \psi_\eps(s,x) ds, \qquad \forall t\geq 0, \; x \in \Omega.
\end{equation}
The convergence in Eq.~\eqref{eq:weak:vp} is straightforward given strong the convergence results of $\vp_\eps$ and weak convergence of $ \psi_\eps$. In Eq.~\eqref{eq:weak:psi} the convergence is immediate for the time derivative and the initial data, but more delicate for the other terms, particularly for the drift term. We rewrite the second term on the right-hand-side of Eq.~\eqref{eq:weak:psi} as
\[
\int_0^T \hspace{-5pt}\int_\Omega 
 \psi_\eps \nabla \chi\cdot {\V}(\psi_\eps,\vp_\eps) \dxi\dti 
 = \int_0^T \hspace{-5pt}\int_\Omega 
\alpha \underbrace{ \psi_\eps \vp_\eps(1-\vp_\eps) \; \nabla \chi \cdot\dx \vp_\eps}_{\mathcal{A}_{\eps}}
-D\underbrace{ \psi_\eps \vp_\eps(1-\vp_\eps) \; \nabla \chi \cdot\dx \psi_\eps}_{\mathcal{B}_\eps} \dxi\dti \,.
\]
From the convergence results detailed in Step 1, the term $\mathcal{A}_{\eps}$ converges strongly to 
$\mathcal{A}=\psi\,\vp (1-\vp)\;\dx\chi\cdot\dx \vp$ because  $\psi_\eps \vp_\eps(1-\vp_\eps)$ converges strongly in $L^{p}((0,T)\times B)$, $1\leq p<3$,  while $\dx \vp_\eps$ converges weakly in $L^{2}((0,T)\times B)$. 
For the other term, we write 
$$
\mathcal{B}_\eps =(1-\vp_\eps) \sqrt{\psi_\eps} \vp_\eps^{\frac 1 4} \; \vp_\eps^{\frac 1 4}\sqrt{\psi_\eps \vp_\eps} \nabla \chi \cdot\dx \psi_\eps.
$$
The term $(1-\vp_\eps) \sqrt{\psi_\eps} \vp_\eps^{\frac 1 4}$ converges strongly in any $L^{p}((0,T)\times B)$, $1\leq p<6$, to $(1-\vp) \sqrt{\psi} \vp^{\frac 1 4}$ using Step~1, while $\vp_\eps^{\frac 1 4} \sqrt{\psi_\eps \vp_\eps} \nabla \chi \cdot\dx \psi_\eps$ is bounded in $L^{2}((0,T)\times B)$ thanks to \eqref{est:FullGrad3:phi}, and thus converges weakly. To identify its limit we write it as
$$
\vp_\eps^{\frac 1 4} \sqrt{\psi_\eps \vp_\eps} \nabla \chi \cdot\dx \psi_\eps = 
\underbrace{ \vp_\eps^{\frac 1 4}\sqrt{\psi_\eps} \nabla \chi \cdot\dx (\sqrt{\vp_\eps}\psi_\eps)}_{\mathcal{C}_{1\eps}}
- \underbrace{\vp_\eps^{\frac 1 4} {\psi_\eps}\,  \sqrt{\psi_\eps}\nabla \chi \cdot \nabla \sqrt{\vp_\eps}}_{\mathcal{C}_{2\eps}}.
$$
The term $\mathcal{C}_{1\eps}$ converges because, as before,  $\vp_\eps^{\frac 1 4} \sqrt{\psi_\eps}$ converges strongly to $\vp^{\frac 1 4} \sqrt{\psi}$ in any $L^{p}((0,T)\times B)$, $1\leq p<6$, while $\nabla \chi \cdot\dx (\sqrt{\vp_\eps}\psi_\eps)$ converges weakly in {$L^{\frac 3 2}((0,T)\times B)$} to $\nabla \chi \cdot\dx (\sqrt{\vp}\psi)$.
The term $\mathcal{C}_{2\eps}$ converges because, as before,  $\vp_\eps^{\frac 1 4} {\psi_\eps}$ converges strongly to $\vp^{\frac 1 4} {\psi}$  in any $L^{p}((0,T)\times B)$, $1\leq p<3$, while $\sqrt{\psi_\eps}\nabla \chi \cdot \nabla \sqrt{\vp_\eps}$ converges weakly in $L^{2}((0,T)\times B)$ to $\sqrt{\psi}\nabla \chi \cdot \nabla \sqrt{\vp}$. 
 

Finally, \REV{on the set $\{\varphi > 0\}$ we have proved pointwise convergence both for $\varphi_\varepsilon$ and $\psi_\varepsilon$, therefore $R_\varepsilon$ converges a.e. to $R=R_0(1-\varphi-\psi)$. So the pointwise limit of $\psi_\varepsilon R_\varepsilon$ is $\psi R$ and it is also the limit in all $L^p$, $p<3$, spaces.  Otherwise, we can only conclude that $R_\varepsilon$ converges weakly to some function $\bar{R}$, with $\bar{R}=R$ on the set $\{ \vp>0 \}$, but we do not know what $\psi_\varepsilon^2$ converges to on the set $\{\vp=0\}$. We thus conclude that} the limit functions \REV{$(\psi,\vp)$} satisfy Eq.~\eqref{eq} in the sense of Definition~\ref{def:weak}
\hfill \qed 
\end{proof}

\REV{\section{Existence of solutions\label{sec:existence} }}

Following the strategy adopted in~\cite{zhigun2016global,zhigun2016global2}, we infer existence of weak solutions of the degenerate problem from an approximating problem with a non-degenerate velocity field, exploiting the regularity established in Sections~\ref{sec:estimates:1}-\ref{sec:compactness}.

For $\delta>0$, let $\psi_\delta(t,\x)\geq0$ and $\vp_\delta(t,\x)\geq0$ satisfy in the weak sense the system
\begin{equation}\label{eq:delta}
\begin{cases}
    \partial_t \psi_\delta +\nabla\cdot \left[ \psi_\delta \big( \delta+\vp_\delta(1-\vp_\delta)\big)\big(\alpha\nabla\vp_\delta-D\nabla\psi_\delta\big)  \right] = \psi_\delta\, R_0(1-\vp_\delta-\psi_\delta),\\
    \partial_t \vp_\delta=-\gamma\vp_\delta\psi_\delta,
\end{cases}
\end{equation}
for $t\in[0,T]$ and $\x\in\Omega\subseteq\mathbb{R}^d$ ($d=1,2,3$), complemented with initial conditions $\psi_\delta(0,\x)=\psi^0_\delta$ and $\vp_\delta(0,\x)=\vp^0_\delta$ satisfying the properties~\eqref{as:1} and~\eqref{as:2}. 
This system is obtained from the degenerate one by relaxing the saturation term~\eqref{def:S} mapping $S\to S_\delta$ with $$S_\delta(\vp):=\delta +S(\vp)>0$$ for all $\varphi\in[0,1]$ -- cf. Figure~\ref{f1} -- or, analogously, $V\to V_\delta$ with $$V_\delta(\psi,\vp):=S_\delta(\vp)\,(\alpha\nabla\vp-D\nabla\psi)\equiv \delta\,(\alpha\nabla\vp-D\nabla\psi)+ V(\psi,\vp)$$ where it is clear that the degeneracy at $\vp=0$ disappears. \\

\subsection{Regularity results and compactness}
Let us extend the results of Sections~\ref{sec:estimates:1}-\ref{sec:compactness} to the non-degenerate problem.

\begin{proposition}[Bounds and compactness in the non-degenerate problem]
    \label{prop3} Let $\Omega\equiv\mathbb{R}^d$, and let $\psi_\delta\geq0$ and $\vp_\delta\geq0$ be solutions of~\eqref{eq:delta} complemented with initial conditions $\psi_\delta(0,\x)=\psi^0_\delta$ and $\vp_\delta(0,\x)=\vp^0_\delta$ satisfying assumptions~\eqref{as:1} and~\eqref{as:2}. 
    Then we have that $\psi_\delta\geq0$ and $\vp_\delta\geq0$ satisfy uniformly in $\delta$ all bounds listed in Propositions~\ref{prop1} and~\ref{prop2}, as well as the following bound for all $\delta>0$ and $T>0$:
    \begin{equation}\label{est:new}
        \delta\,\psi_\delta\,|D\nabla\psi_\delta-\alpha\nabla\vp_\delta|^2\in L^1 ((0,T)\times\Omega
    \end{equation}
    Moreover, we have local compactness of $\vp_\delta$ and $\vp^{\frac{1}{2}}_\delta \psi_\delta$ in space and time. 
\end{proposition}
\begin{proof}

\noindent \underline{Step 1. The bounds in Proposition~\ref{prop1} and the new estimate.}
The proof of bounds~\eqref{est:phiLinf}, \eqref{est:psiL1} and \eqref{est:psiLog}-\eqref{est:gradient} follows the same steps detailed in the proof of Proposition~\ref{prop1}. 
Meanwhile the energy has an extra term, yielding
\begin{equation*}
\begin{split}
    \frac{d}{dt} \intx\psi\left(\frac D 2 \psi_\delta - \al \vp_\delta\right)\dxi +\delta \intx \psi_\delta|D\nabla\psi_\delta-\alpha\nabla\vp_\delta|^2{\rm d}\x \leq \;& C_2 \intx\psi_\delta\left(\frac D 2 \psi_\delta - \al \vp_\delta\right)\dxi \\
    &+\alpha (C_2+R_0) \intx\psi_\delta\dxi ,
\end{split}
\end{equation*}
from which we can again obtain~\eqref{est:energy}. Moreover, integrating in time, we conclude
\begin{equation*}
        \delta \int_0^T\int_\Omega \psi_\delta\,|D\nabla\psi_\delta-\alpha\nabla\vp_\delta|^2\,{\rm d}\x\,{\rm d}t\leq C_{12}(T) \,.
\end{equation*}
thanks to~\eqref{as:1}, \eqref{est:psiL2} and~\eqref{est:energy}, i.e. we have the new bound~\eqref{est:new}. 
This new estimate allows us to conclude the proof of bound~\eqref{est:psiX2}. 
In fact, the second moment also has an additional term
\begin{align*}
\frac{d}{dt} \intx\frac{|\x|^2}{2} \psi_\delta(t,\x) \dxi = & \intx |\x|\, \psi_\delta \vp_\delta (1-\vp_\delta) |\al \dx \vp_\delta -D \dx \psi_\delta| \dxi+ \intx \frac{|\x|^2}{2}  \psi_\delta R_0(1-\psi_\delta-\vp_\delta) \dxi
\\
&+\delta \intx |\x|\,\psi_\delta \,|D\nabla\psi_\delta-\alpha\nabla\vp_\delta|\dxi 
\end{align*}
to which we can apply the Cauchy-Schwarz and the AM-GM inequalities, as we did to the first term (see Step 3 in proof of Proposition~\ref{prop1}), to obtain
\begin{align*}
\frac{d}{dt} \intx\frac{|\x|^2}{2} \psi_\delta(t,\x) \dxi \leq&   \intx\frac{|\x|^2}{2} \psi_\delta \vp_\delta (1-\vp_\delta) \dxi+ \frac 12  \intx\psi_\delta \vp_\delta (1-\vp_\delta) |\al \dx \vp_\delta -D \dx \psi_\delta|^2\dxi   \\[3pt]
&+ \delta \intx \frac{|\x|^2}{2}\psi_\delta\dxi \;+ \frac{\delta}{2}\intx\psi |\alpha\nabla\vp_\delta-D\nabla\psi_\delta|^2\dxi  +
\intx\frac{|\x|^2}{2}  \psi_\delta R_0 \dxi \,.
\end{align*}
We now can apply the Gronwall lemma, thanks to~\eqref{as:1}, \eqref{est:gradient} and the new bound~\eqref{est:new}, and conclude~\eqref{est:psiX2}.\\

\noindent \underline{Step 2. The bounds in Proposition~\ref{prop2}.} The proof of the stronger gradient control~\eqref{est:FullGrad:psi} now needs to also rely on the Gronwall lemma. In fact the entropy has an extra term on the right-hand-side of~\eqref{est:entropy} which can be bounded by
\begin{align*}
    \delta \intx \nabla\psi_\delta (\alpha \nabla\vp_\delta-D\nabla\psi_\delta)\dxi =& -\delta D\intx |\nabla\psi_\delta|^2\dxi +\intx \left( \nabla\psi_\delta\delta^{\frac{1}{2}}D^{\frac{1}{2}} \right) \left(\frac{\delta^{\frac{1}{2}}}{D^{\frac{1}{2}}}\alpha\nabla\vp_\delta\right)\dxi\\
    \leq&-\delta\frac{D}{2}\intx |\nabla\psi_\delta|^2\dxi +\delta \frac{\alpha}{2D}\intx |\nabla\vp|^2\dxi\,.
\end{align*}
Then combining the entropy with $\partial_t \left[\frac{|\dx \vp|^2}{2} (1-\vp)\right]$, the above bound and those detailed in Step 1 of Proposition~\ref{prop2} yield
\begin{align*}
\frac{d}{dt} \intx \bigg[\frac \al \gamma \frac{|\dx \vp|^2}{2} (1-\vp) +  \psi  \ln(\psi) \bigg]\dxi \leq \delta \frac{\alpha}{2D}\intx |\nabla\vp|^2\dxi + R_0  \intx \psi \dxi +R_0 \intx  \psi|\ln(\psi)| \dxi  \,.
\end{align*} 
From here we can apply the Gronwall lemma thanks to bounds~\eqref{est:psiL1}, \eqref{est:psiLog} and~\eqref{as:2}, concluding~\eqref{est:changes}.  
Then bound~\eqref{est:FullGrad:psi} is obtained from the same calculation, as outlined in Step 1 of Proposition\ref{prop1}, knowing~\eqref{est:changes}. The proof of bounds~\eqref{est:compact:phi}, \eqref{est:FullGrad3:phi}-\eqref{est:compact:psi} remains unchanged.\\

\noindent \underline{Step 3. Local compactness.} The proof of local compactness follows the same steps as in Section~\ref{sec:compactness}, with space compactness and bound~\eqref{est:timecompact:vp} retrieved analogously. To gain time-compactness of $\vp^{\frac{1}{2}}_\delta\psi_\delta$ we need to bound extra terms appearing in~\eqref{eqbis} which reads as
\begin{align*}
\dt (\vp_\delta^{\frac 1 2} \psi_\delta) =&-  \dx \cdot \left[\, \vp_\delta^{\frac 1 2} \psi_\delta\,  {\V}(\psi_\delta,\vp_\delta)\,\right]+ \;  \psi_\delta \dx \vp_\delta^{\frac 1 2}\cdot {\V}(\psi_\delta,\vp_\delta) + \psi_\delta  \vp_\delta^{\frac 1 2} R(\psi_\delta,\vp_\delta) -\frac \gamma 2 \vp_\delta^{\frac 1 2} \psi_\delta^2 \\
&-\nabla\cdot\left[\,\delta\, \vp_\delta^{\frac 1 2} \psi_\delta\, \left( \alpha\nabla\vp_\delta-D\nabla\psi_\delta\right)\,\right] + \,\delta \,\psi_\delta \dx \vp_\delta^{\frac 1 2}\cdot \left( \alpha\nabla\vp_\delta-D\nabla\psi_\delta\right)\,.
\end{align*}
Using the Cauchy-Schwarz inequality we see that
\begin{align*}
   \delta \intt\hspace{-5pt} \intx \vp_\delta^{\frac 1 2} \psi_\delta\, | \alpha\nabla\vp_\delta-D\nabla\psi_\delta|\dxi {\rm d}t \leq \delta \left[ \intt\hspace{-5pt} \intx \psi_\delta\dxi{\rm d}t\; \intt\hspace{-5pt} \intx \vp_\delta\psi_\delta| \alpha\nabla\vp_\delta-D\nabla\psi_\delta|^2\dxi {\rm d}t \right]^{\frac{1}{2}}
\end{align*}
is bounded in $L^1_{\text{loc}}((0,T),\Omega)$, as the first integral is bounded by~\eqref{est:psiL1} and the second can be bounded by~\eqref{est:gradient} knowing~\eqref{est:phiLinf} under assumption~\eqref{as:2}. Similarly, we have that
\begin{align*}
\delta \intt\hspace{-5pt} \intx | \dx \vp_\delta^{\frac 1 2}\cdot \psi_\delta\left( \alpha\nabla\vp_\delta-D\nabla\psi_\delta\right)|\dxi {\rm d}t \leq \delta^\frac{1}{2}\left[ \intt\hspace{-5pt} \intx \psi_\delta|\nabla \sqrt{\psi_\delta}|^2\dxi {\rm d}t \; \delta\intt\hspace{-5pt} \intx \psi_\delta|\alpha\nabla\vp_\delta-D\nabla\psi_\delta|^2\dxi {\rm d}t\right]^{\frac{1}{2}}
\end{align*}
is bounded in $L^1_{\text{loc}}((0,T),\Omega)$, since the first integral is bounded by~\eqref{est:FullGrad:psi} and the second by the new bound~\eqref{est:new}. Then we can apply the Aubins-Lions-Simons lemma, concluding the proof. \hfill \qed
\end{proof}

\subsection{Construction of solutions}

The non-degenerate problem~\eqref{eq:delta}, still highly nonlinear, can be handled following the arguments in~\cite{zhigun2016global2,zhigun2018strongly,zhigun2016global} and existence of global strong solutions follows. The solutions satisfy the properties of Proposition~\ref{prop3} uniformly in $\delta$.

\begin{theorem}[Global existence of solutions]
Let $\Omega\equiv\mathbb{R}^d$ $(\psi_\delta,\vp_\delta)$ denote a global solution of the non-degenerate problem~\eqref{eq:delta} complemented with $\psi_\delta(0,\x)=\psi^0_\delta$ and $\vp_\delta(0,\x)=\vp^0_\delta$ satisfying the properties~\eqref{as:1} and~\eqref{as:2}. Then, as $\delta\to 0$, $(\psi_\delta,\vp_\delta)$ admits a subsequence converging to the pair $(\psi,\vp)$, a weak solution of the degenerate problem~\eqref{eq}-\eqref{ic} in the sense of Definition~\ref{def:weak}. 
\end{theorem}

\begin{proof}
Taking for granted by~\cite{zhigun2016global2,zhigun2018strongly,zhigun2016global} that there exists a global solution of the non-degenerate problem~\eqref{eq:delta}, under assumptions~\eqref{as:1} and~\eqref{as:2}, we know that it satisfies all properties of Proposition~\ref{prop3}. 
    Then, since all properties are satisfied uniformly in $\delta$, letting $\delta \to 0$ we can pass to the limit along with subsequences similarly to as done in Section~\ref{sec:stability}, and the limit pair $(\psi,\vp)$ satisfies all properties of Propositions~\ref{prop1} and~\ref{prop2}. 
    This shows the global existence of weak solutions of the degenerate problem~\eqref{eq}. 
    \hfill \qed
\end{proof}

\section{Discussion}
\label{sec:discuss}

We focused on a strongly degenerate nonlinear diffusion and haptotaxis model of cancer invasion, implementing the ECM-dependent saturation of the speed of cell motion due to the physical limits of cell migration. The resulting equation for the cell volume fraction dynamics exhibits two regimes of degeneracy of the velocity, which can make it particularly challenging to obtain sufficient regularity results to investigate the well-posedness of the system. We here devised specific regularity estimates, in finite time, that are compatible with these degeneracies. After some regularity results under physical bounds on the initial ECM volume fraction $\vp_0$ (i.e. $0\leq\vp_0\leq 1$), within which both regimes of degeneracy may be encountered, we obtained additional regularity estimates under the stricter bounds $0\leq\vp_0\leq \frac 2 3$, focusing on the vacuum observed at $\vp=0$. 
Based on these estimates and the associated local compactness results obtained, we proved the stability  of -- appropriately defined -- weak solutions with respect to perturbations of the initial data. This confers robustness to the proposed cancer invasion model, despite the strong degeneracy. 
%
\REV{Finally, global existence of weak solutions was proved from an approximate non-degenerate problem, exploiting simple extensions of the preceding regularity results. Uniqueness and large-time behaviour of solutions of the strongly degenerate system are open problems.
}

Indeed, we obtained the strongest results under the stricter bound on the initial ECM volume fraction $0<\vp_0\leq \frac 2 3 $. This is not an unreasonable assumption, given the physical meaning of the variable $\vp$. On one hand, the ECM is a porous medium and in practice its volume fraction will not be close to 1. On the other hand, it is known that ECM is generally also present in the bulk of solid tumours, as captured by the strict positivity of $\vp$. Nevertheless, the extension of the results presented in this work to the most general case with $0\leq\vp^0\leq1$ should still be pursued. This is particularly relevant in view of the fact that here we chose to define $S(\vp)$, i.e. the modulation of the cell motility regulated by the ECM volume fraction, as a simplified version of the one derived in~\cite{giverso2018nucleus}, i.e.
\begin{equation*}\label{def:DM}
    S(\vp) = 4\frac{\left[ (\vp-\vp_{min})(\vp_{th}-\vp) \right]_+}{(\vp_{th}-\vp_{min})^2} \,.
\end{equation*}
Here, $\vp_{min}>0$ is the minimum ECM volume fraction below which cell movement through the ECM is hindered, and $\vp_{th}>\vp_0$ is a threshold ECM volume fraction above which cells can no longer pass through the ECM pores. Definition~\eqref{def:S} consists in the specific case in which $\vp_{min}=0$ and $\vp_{th}=1$, but different values will be more biologically relevant, in which case the degeneracy will occur for intermediate values of $\vp$ \REV{-- see also Figure~\ref{f2}}.
\begin{figure}
    \centering
    \includegraphics[width=\linewidth]{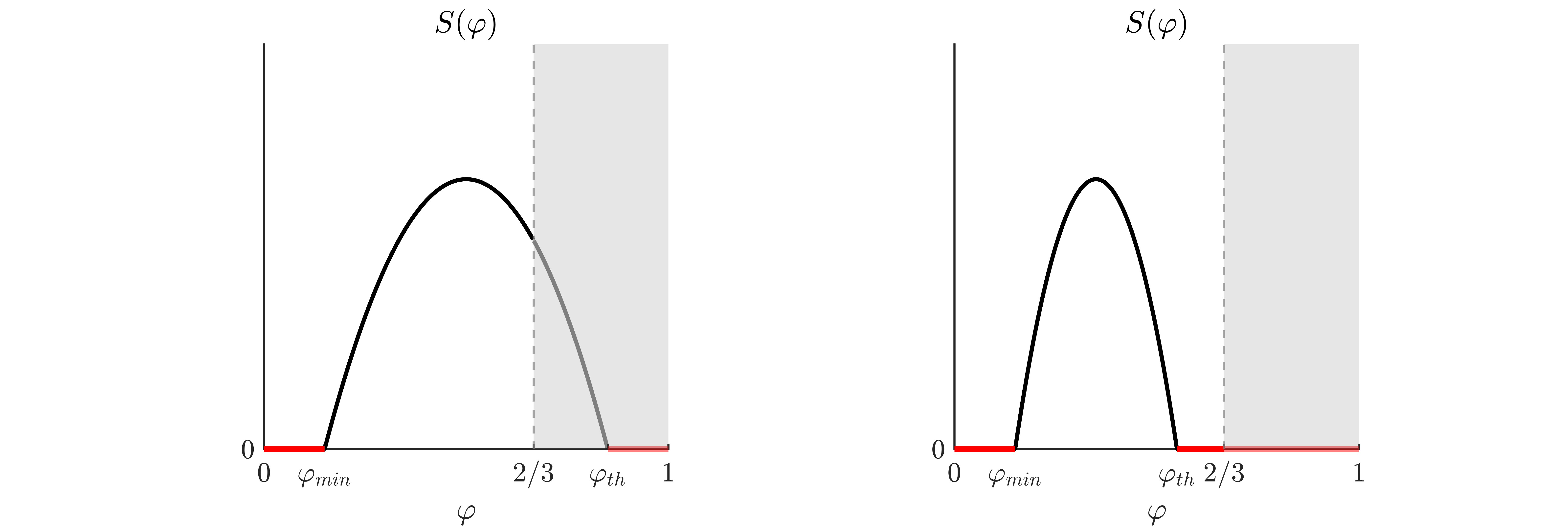}
    \caption{\REV{Alternative behaviours of the function $S(\vp)$ based on definition~\eqref{def:DM} proposed in~\cite{giverso2018nucleus,loy2020modelling}, for $\vp_{th}>\frac{2}{3}$ (left) or $\vp_{th}<\frac{2}{3}$ (right), displaying wider regions of strong degeneracy (in red) of the velocity $\bm V$ defined in~\eqref{def:V}.   }}
    \label{f2}
\end{figure}

\REV{%
Finally, we remark that discretised versions of the estimates obtained in this work will be exploited to design and analyse a robust numerical scheme based on finite volume approximations to simulate the system. The scheme should effectively address the stiffness of the problem by preserving non-negativity, avoiding spurious oscillations, and accurately capturing the sharp gradients of solutions typical of cell migration PDE models~\cite{chertock2008second,kolbe2014numerical,lorenzi2025phenotype}, particularly critical for nonlinear and degenerate parabolic equations~\cite{david2022asymptotic,liu2011high}, meanwhile conserving the qualitative properties of the system in the degenerate regimes~\cite{bessemoulin2012finite2,bessemoulin2012finite}.}

\section*{Acknowledgements}
This project has received funding from the European Union's Horizon 2020 research and innovation programme under the Marie Skłodowska-Curie grant agreement No 945298.  This project has received funding from the Paris Region under the Paris Region fellowship Programme.


\begin{thebibliography}{10}

\bibitem{almeida2025mathematical}
{\sc L.~Almeida, A.~Poulain, A.~Pourtier, and C.~Villa}, {\em Mathematical
  modelling of the contribution of senescent fibroblasts to basement membrane
  digestion during carcinoma invasion}, hal-04574340,  (2025).

\bibitem{anderson2000mathematical}
{\sc A.~R. Anderson, M.~A. Chaplain, E.~L. Newman, R.~J. Steele, and A.~M.
  Thompson}, {\em Mathematical modelling of tumour invasion and metastasis},
  Computational and Mathematical Methods in Medicine, 2 (2000), pp.~129--154.

\bibitem{arduino2015multiphase}
{\sc A.~Arduino and L.~Preziosi}, {\em A multiphase model of tumour segregation
  in situ by a heterogeneous extracellular matrix}, International Journal of
  Non-Linear Mechanics, 75 (2015), pp.~22--30.

\bibitem{aznavoorian1990signal}
{\sc S.~Aznavoorian, M.~L. Stracke, H.~Krutzsch, E.~Schiffmann, and L.~A.
  Liotta}, {\em Signal transduction for chemotaxis and haptotaxis by matrix
  molecules in tumor cells}, The Journal of Cell Biology, 110 (1990),
  pp.~1427--1438.

\bibitem{bellomo2015toward}
{\sc N.~Bellomo, A.~Bellouquid, Y.~Tao, and M.~Winkler}, {\em Toward a
  mathematical theory of {K}eller--{S}egel models of pattern formation in
  biological tissues}, Mathematical Models and Methods in Applied Sciences, 25
  (2015), pp.~1663--1763.

\bibitem{bessemoulin2012finite2}
{\sc M.~Bessemoulin-Chatard}, {\em {A finite volume scheme for
  convection--diffusion equations with nonlinear diffusion derived from the
  Scharfetter--Gummel scheme}}, Numerische Mathematik, 121 (2012),
  pp.~637--670.

\bibitem{bessemoulin2012finite}
{\sc M.~Bessemoulin-Chatard and F.~Filbet}, {\em A finite volume scheme for
  nonlinear degenerate parabolic equations}, SIAM Journal on Scientific
  Computing, 34 (2012), pp.~B559--B583.

\bibitem{chaplain2006mathematical}
{\sc M.~A. Chaplain and G.~Lolas}, {\em Mathematical modelling of cancer
  invasion of tissue: dynamic heterogeneity}, Networks and Heterogeneous Media,
  1 (2006), pp.~399--439.

\bibitem{chertock2008second}
{\sc A.~Chertock and A.~Kurganov}, {\em A second-order positivity preserving
  central-upwind scheme for chemotaxis and haptotaxis models}, Numerische
  Mathematik, 111 (2008), pp.~169--205.

\bibitem{dai2022global}
{\sc F.~Dai and B.~Liu}, {\em Global weak solutions in a three-dimensional
  two-species cancer invasion haptotaxis model without cell proliferation},
  Journal of Mathematical Physics, 63 (2022).

\bibitem{dai2023boundedness}
{\sc F.~Dai and L.~Ma}, {\em Boundedness in a two-dimensional two-species
  cancer invasion haptotaxis model without cell proliferation}, Zeitschrift
  f{\"u}r angewandte Mathematik und Physik, 74 (2023), p.~54.

\bibitem{david2022asymptotic}
{\sc N.~David and X.~Ruan}, {\em An asymptotic preserving scheme for a tumor
  growth model of porous medium type}, ESAIM: Mathematical Modelling and
  Numerical Analysis, 56 (2022), pp.~121--150.

\bibitem{doumic2024multispecies}
{\sc M.~Doumic, S.~Hecht, B.~Perthame, and D.~Peurichard}, {\em Multispecies
  cross-diffusions: from a nonlocal mean-field to a porous medium system
  without self-diffusion}, Journal of Differential Equations, 389 (2024),
  pp.~228--256.

\bibitem{friedl2003tumour}
{\sc P.~Friedl and K.~Wolf}, {\em Tumour-cell invasion and migration: diversity
  and escape mechanisms}, Nature Reviews Cancer, 3 (2003), pp.~362--374.

\bibitem{giverso2018nucleus}
{\sc C.~Giverso, A.~Arduino, and L.~Preziosi}, {\em {How nucleus mechanics and
  ECM microstructure influence the invasion of single cells and multicellular
  aggregates}}, Bulletin of Mathematical Biology, 80 (2018), pp.~1017--1045.

\bibitem{hotary2006cancer}
{\sc K.~Hotary, X.-Y. Li, E.~Allen, S.~L. Stevens, and S.~J. Weiss}, {\em A
  cancer cell metalloprotease triad regulates the basement membrane
  transmigration program}, Genes \& development, 20 (2006), pp.~2673--2686.

\bibitem{jin2024critical}
{\sc C.~Jin}, {\em Critical exponent to a cancer invasion model with nonlinear
  diffusion}, Journal of Mathematical Physics, 65 (2024).

\bibitem{jin2024roles}
\leavevmode\vrule height 2pt depth -1.6pt width 23pt, {\em The roles of
  nonlinear diffusion, haptotaxis and ecm remodelling in determining the global
  solvability of a cancer invasion model}, Proceedings of the Royal Society of
  Edinburgh Section A: Mathematics,  (2024), pp.~1--32.

\bibitem{kolbe2014numerical}
{\sc N.~Kolbe, J.~Katuchova, N.~Sfakianakis, N.~Hellmann, and
  M.~Lukacova-Medvidova}, {\em Numerical study of cancer cell invasion dynamics
  using adaptive mesh refinement: the urokinase model}, Applied Mathematics and
  Computation, 273 (2014).

\bibitem{liu2011high}
{\sc Y.~Liu, C.-W. Shu, and M.~Zhang}, {\em {High order finite difference WENO
  schemes for nonlinear degenerate parabolic equations}}, SIAM Journal on
  Scientific Computing, 33 (2011), pp.~939--965.

\bibitem{lorenzi2025phenotype}
{\sc T.~Lorenzi, K.~J. Painter, and C.~Villa}, {\em Phenotype structuring in
  collective cell migration: a tutorial of mathematical models and methods},
  Journal of Mathematical Biology, 90 (2025), p.~61.

\bibitem{loy2020modelling}
{\sc N.~Loy and L.~Preziosi}, {\em Modelling physical limits of migration by a
  kinetic model with non-local sensing}, Journal of Mathematical Biology, 80
  (2020), pp.~1759--1801.

\bibitem{marciniak2010boundedness}
{\sc A.~Marciniak-Czochra and M.~Ptashnyk}, {\em Boundedness of solutions of a
  haptotaxis model}, Mathematical Models and Methods in Applied Sciences, 20
  (2010), pp.~449--476.

\bibitem{perumpanani1999extracellular}
{\sc A.~Perumpanani and H.~Byrne}, {\em Extracellular matrix concentration
  exerts selection pressure on invasive cells}, European Journal of Cancer, 35
  (1999), pp.~1274--1280.

\bibitem{poincloux2009matrix}
{\sc R.~Poincloux, F.~Liz{\'a}rraga, and P.~Chavrier}, {\em {Matrix invasion by
  tumour cells: a focus on MT1-MMP trafficking to invadopodia}}, Journal of
  cell science, 122 (2009), pp.~3015--3024.

\bibitem{scianna2012cellular}
{\sc M.~Scianna, L.~Preziosi, and K.~Wolf}, {\em A cellular potts model
  simulating cell migration on and in matrix environments}, Mathematical
  Biosciences \& Engineering, 10 (2012), pp.~235--261.

\bibitem{sfakianakis2020mathematical}
{\sc N.~Sfakianakis and M.~A. Chaplain}, {\em Mathematical modelling of cancer
  invasion: a review}, in International Conference by Center for Mathematical
  Modeling and Data Science, Osaka University, Springer, 2020, pp.~153--172.

\bibitem{shangerganesh2019finite}
{\sc L.~Shangerganesh, N.~Nyamoradi, G.~Sathishkumar, and S.~Karthikeyan}, {\em
  Finite-time blow-up of solutions to a cancer invasion mathematical model with
  haptotaxis effects}, Computers \& Mathematics with Applications, 77 (2019),
  pp.~2242--2254.

\bibitem{simon1986compact}
{\sc J.~Simon}, {\em Compact sets in the space ${L}^p ({0}, {T}; {B})$}, Annali
  di Matematica Pura ed Applicata, 146 (1986), pp.~65--96.

\bibitem{tao2010density}
{\sc Y.~Tao and C.~Cui}, {\em A density-dependent chemotaxis--haptotaxis system
  modeling cancer invasion}, Journal of Mathematical Analysis and Applications,
  367 (2010), pp.~612--624.

\bibitem{tao2007global}
{\sc Y.~Tao and G.~Zhu}, {\em Global solution to a model of tumor invasion},
  Applied Mathematical Sciences, 1 (2007), pp.~2385--2398.

\bibitem{walker2007global}
{\sc C.~Walker and G.~F. Webb}, {\em Global existence of classical solutions
  for a haptotaxis model}, SIAM Journal on Mathematical Analysis, 38 (2007),
  pp.~1694--1713.

\bibitem{wang2020review}
{\sc Y.~Wang}, {\em A review on the qualitative behavior of solutions in some
  chemotaxis--haptotaxis models of cancer invasion}, Mathematics, 8 (2020),
  p.~1464.

\bibitem{wolf2013physical}
{\sc K.~Wolf, M.~Te~Lindert, M.~Krause, S.~Alexander, J.~Te~Riet, A.~L. Willis,
  R.~M. Hoffman, C.~G. Figdor, S.~J. Weiss, and P.~Friedl}, {\em Physical
  limits of cell migration: control by ecm space and nuclear deformation and
  tuning by proteolysis and traction force}, Journal of Cell Biology, 201
  (2013), pp.~1069--1084.

\bibitem{zhigun2016global2}
{\sc A.~Zhigun, C.~Surulescu, and A.~Hunt}, {\em Global existence for a
  degenerate haptotaxis model of tumor invasion under the go-or-grow dichotomy
  hypothesis}, arXiv preprint arXiv:1605.09226,  (2016).

\bibitem{zhigun2018strongly}
\leavevmode\vrule height 2pt depth -1.6pt width 23pt, {\em A strongly
  degenerate diffusion-haptotaxis model of tumour invasion under the go-or-grow
  dichotomy hypothesis}, Mathematical Methods in the Applied Sciences, 41
  (2018), pp.~2403--2428.

\bibitem{zhigun2016global}
{\sc A.~Zhigun, C.~Surulescu, and A.~Uatay}, {\em Global existence for a
  degenerate haptotaxis model of cancer invasion}, Zeitschrift f{\"u}r
  angewandte Mathematik und Physik, 67 (2016), pp.~1--29.

\end{thebibliography}
\end{document}